\documentclass[11pt,letterpaper]{amsart}

\usepackage[T1]{fontenc}
\usepackage{lmodern}
\usepackage{hyperref}
\usepackage{etex}
\usepackage[shortlabels]{enumitem}
\usepackage{amsmath}
\usepackage{amsxtra}
\usepackage{amscd}
\usepackage{amsthm}
\usepackage{amsfonts}
\usepackage{amssymb}
\usepackage{eucal}
\usepackage[all]{xy}
\usepackage{graphicx}
\usepackage{tikz-cd}
\usepackage{mathrsfs}
\usepackage{subfiles}
\usepackage{mathpazo}
\usepackage{mathtools}
\usepackage[colorinlistoftodos, textsize=tiny]{todonotes}
\usepackage{morefloats}
\usepackage{pdfpages}
\usepackage{thm-restate}
\usepackage[utf8]{inputenc}
\usepackage{epigraph}
\usepackage{csquotes}
\usepackage[margin=1in]{geometry}
\usepackage{adjustbox}
\usepackage{scalerel}
\usepackage{stackengine}
\stackMath
\newcommand\reallywidehat[1]{%
\savestack{\tmpbox}{\stretchto{%
  \scaleto{%
    \scalerel*[\widthof{\ensuremath{#1}}]{\kern-.6pt\bigwedge\kern-.6pt}%
    {\rule[-\textheight/2]{1ex}{\textheight}}%
  }{\textheight}%
}{0.5ex}}%
\stackon[1pt]{#1}{\tmpbox}%
}
\graphicspath{ {images/} }

\RequirePackage{color}
\definecolor{myred}{rgb}{0.75,0,0}
\definecolor{mygreen}{rgb}{0,0.5,0}
\definecolor{myblue}{rgb}{0,0,0.65}

\hypersetup{citecolor=blue}
\usepackage{tikz}
\usetikzlibrary{matrix,arrows,decorations.pathmorphing}

\theoremstyle{plain}
\newtheorem{theorem}[subsection]{Theorem}

\newtheorem{proposition}[subsection]{Proposition}
\newtheorem{lemma}[subsection]{Lemma}

\theoremstyle{definition}
\newtheorem{definition}[subsection]{Definition}
\newtheorem{remark}[subsection]{Remark}

\theoremstyle{remark}

\numberwithin{equation}{subsection}
\newcommand\nc{\newcommand}
\nc\on{\operatorname}
\nc\renc{\renewcommand}

\newcommand\bc{\mathbb C}

\newcommand\bg{\mathbb G}
\newcommand\bh{\mathbb H}

\newcommand\bl{\mathbb L}

\newcommand\bp{\mathbb P}
\newcommand\bq{\mathbb Q}

\newcommand\bv{\mathbb V}
\newcommand\bw{\mathbb W}

\newcommand\bz{\mathbb Z}

\newcommand*{\shom}{\mathscr{H}\kern -.5pt om}
\newcommand*{\stor}{\mathscr{T}\kern -.5pt or}
\newcommand*{\sext}{\mathscr{E}\kern -.5pt xt}

\makeatletter
\providecommand\@dotsep{5}
\renewcommand{\listoftodos}[1][\@todonotes@todolistname]{%
\@starttoc{tdo}{#1}}
\makeatother

\makeatletter
\newcommand{\customlabel}[2]{\protected@write \@auxout {}{\string \newlabel {#1}{{#2}{\thepage}{#2}{#1}{}} }\hypertarget{#1}{#2}}

\DeclareMathOperator\id{id}

\renewcommand\hom{\mathrm{Hom}}

\DeclareMathOperator\End{End}
\DeclareMathOperator\sym{Sym}
\DeclareMathOperator\pgl{PGL}

\DeclareMathOperator\sgn{sgn}

\DeclareMathOperator\ab{ab}

\DeclareMathOperator\gal{Gal}
\DeclareMathOperator\tr{tr}

\DeclareMathOperator\aut{Aut}
\DeclareMathOperator\gl{GL}

\renewcommand\sl{\mathrm{SL}}

\DeclareMathOperator\so{SO}

\DeclareMathOperator\ad{ad}

\DeclareMathOperator\diag{diag}

\DeclareMathOperator\orb{orb}

\DeclareMathOperator{\betti}{Betti}
\DeclareMathOperator{\higgs}{Higgs}

\DeclareMathOperator{\coarse}{coarse}

\DeclareFontFamily{U}{wncy}{}
\DeclareFontShape{U}{wncy}{m}{n}{<->wncyr10}{}
\DeclareSymbolFont{mcy}{U}{wncy}{m}{n}
\DeclareMathSymbol{\Sha}{\mathord}{mcy}{"58}

\def\listtodoname{List of Todos}
\def\listoftodos{\@starttoc{tdo}\listtodoname}

\title{Privileged $\sl_3$-local systems are motivic}

\author{Julia Meng, Charlie Wu}

\usepackage{microtype}

\begin{document}

\begin{abstract}
Let $X$ be a smooth projective variety over the complex numbers, and let $\bv$ be an irreducible privileged $\sl_3$-local system on $X$ in the sense of Esnault and Groechenig. We show that $\bv$ is of geometric origin. We also give an example of a privileged but non-rigid $\sl_3$-local system.
\end{abstract}

\maketitle

\setcounter{tocdepth}{1}
\tableofcontents

\section{Introduction}\label{sec:introduction}

The goal of this note is to show that privileged rank $3$ local systems with trivial determinant on smooth projective varieties over the complex numbers (in the sense of Esnault and Groechenig \cite[Definition 2.2]{Esnault_Groechenig}) are of geometric origin.

\begin{definition}\label{def:geometric}
Let $X$ be a smooth variety over the complex numbers and $\bv$ a $\bc$-local system on $X$. We say that $\bv$ is \emph{of geometric origin} if there is a Zariski-dense open subset $U \subseteq X$, a smooth proper map $f: Y \to U$, and an integer $i \geq 0$ such that $\bv|_U$ is isomorphic to a sub-quotient of $R^i f_* \bc$.
\end{definition}

Let $X$ be a smooth projective variety over $\bc$, and $\bl$ a rank one torsion local system. We write $M_{\betti}(X, r, \bl)$ for the moduli space of irreducible rank $r$ local systems on $X$ whose determinant is isomorphic to $\bl$. A \emph{rigid} local system is an isolated point in $M_{\betti}(X, r, \bl)$, and a cohomologically rigid local system is a reduced isolated point in $M_{\betti}(X, r, \bl)$. The local systems we work with are quasi-privileged local systems, which are a generalization of cohomologically rigid local systems.

\begin{definition}\label{def:quasi-privileged}
Let $\bv$ be an irreducible rank $r$ local system on $X$ with determinant $\bl$. We say that $\bv$ is \emph{quasi-privileged} if for all but finitely many (isomorphism classes of) irreducible local systems $\bw \in M_{\betti}(X,r, \bl)$,
\begin{align*}
\dim H^1(X, \End^{0}(\bv)) \neq \dim H^1(X, \End^{0}(\bw)).
\end{align*} 
\end{definition}

\begin{remark}
This is a slight generalization of the notion of \emph{privileged} local systems, first defined and studied by Esnault and Groechenig \cite{Esnault_Groechenig}. A cohomologically rigid local system $\bv$ is quasi-privileged, since the tangent space to $M_{\betti}(X,r,\bl)$ at $\bv$ is isomorphic to $H^1(X, \End^{0}(\bv))$.
\end{remark}

The quasi-privileged local systems enjoy many of the same properties as cohomologically rigid local systems which we will discuss in \autoref{prop:privileged galois}. As a special case, see \cite{Esnault_Groechenig} and \cite{Esnault_Groechenig_Rigid} where Esnault and Groechenig show that both cohomologically rigid and privileged local systems underlie variations of Hodge structure and are defined over rings of integers of number fields.

Our main results are as follows:
\begin{theorem}\label{thm:quasi-privileged implies motivic}
Let $X$ be a connected smooth projective variety over $\bc$, and let $\bv$ be an irreducible rank $3$ local system on $X$ with trivial determinant. If $\bv$ is quasi-privileged, then $\bv$ is of geometric origin.
\end{theorem}

We also give an example of a privileged $\sl_3$-local system on a smooth projective variety which is \emph{not} rigid. To our knowledge, this is the first example of a privileged local system which is not rigid, and the construction is heavily inspired by the work of \cite{dejong2022rigidnoncohomologicallyrigidlocal} where de Jong, Esnault, and Groechenig give an example of a rigid but not cohomologically rigid local system. Like \cite{dejong2022rigidnoncohomologicallyrigidlocal}, our construction is motivic in nature (which should not be surprising due to \autoref{thm:quasi-privileged implies motivic}).

\begin{theorem}\label{thm:privileged non-rigid}
There is a smooth projective variety $X$ and an irreducible $\sl_3$-local system $\bv$ on $X$ which is privileged but not rigid.
\end{theorem}

\subsection{Methods}\label{subsec:methods}
Our methods to prove \autoref{thm:quasi-privileged implies motivic} are heavily inspired by the work of Langer-Simpson \cite{Langer_Simpson}. In fact, they observe that it may be possible to replace rigidity with a zero-dimensional stratum of a stratification by dimensions of cohomology groups \cite[top of page 37]{Langer_Simpson} and obtain results similar to those in their work.

Their main theorem is \cite[Theorem 1.3]{Langer_Simpson}:
\begin{theorem}\label{thm:rank 3 rigid}
Let $X$ be a smooth projective variety. Then, every irreducible integral rigid representation $\rho: \pi_1(X) \to \sl_3(\bc)$ is of geometric origin.
\end{theorem}
Their argument reduces the problem to considering rigid local systems underlying variations of Hodge structure of type $(1,1,1)$ which projectively factor through an orbicurve. Then due to results of Katz on rigid local systems over curves \cite[Theorem 8.4.1]{Katz_1996}, Langer and Simpson conclude that the rigid local system is of geometric origin.

Our strategy is similar. We show in \autoref{prop:privileged galois} that quasi-privileged local systems and their Galois conjugates have the properties we expect  local systems of geometric origin to have, namely that they underlie variations of Hodge structure and are integral. Then, we break up the argument into three cases.

\begin{enumerate}
\item[Case $1$:] No Galois conjugate $\sigma \cdot \bv$ underlies a variation of Hodge structure whose Hodge numbers are $(1,1,1)$,
\item[Case $2$:] Some Galois conjugate $\sigma \cdot \bv$ underlies a variation of Hodge structure whose Hodge numbers are $(1,1,1)$, but it \emph{does not} have Zariski-dense monodromy in $\sl_3(\bc)$,
\item[Case $3$:] Some Galois conjugate $\sigma \cdot \bv$ underlies a variation of Hodge structure whose Hodge numbers are $(1,1,1)$, and it has Zariski-dense monodromy in $\sl_3(\bc)$.
\end{enumerate}
We handle Case $1$ in \autoref{prop: privileged (1,1,1)} and Case $2$ in \autoref{prop:sym is motivic}. The main content of the paper is to handle Case $3$, where we show in \autoref{thm:quasi-priviliged implies rigid} that the local systems satisfying the hypotheses of Case $3$ are cohomologically rigid. Then, applying \autoref{thm:rank 3 rigid} and integrality of $\bv$ tells us that $\bv$ is of geometric origin.

\begin{remark}
Esnault and Groechenig show that privileged local systems are integral \cite[Theorem 2.7]{Esnault_Groechenig}. The same methods also imply that quasi-privileged local systems are integral, so we do not need any integrality hypotheses on \autoref{thm:quasi-privileged implies motivic}.
\end{remark}

\subsection{Structure of the paper}
In \autoref{sec:background}, we briefly give some background on non-abelian Hodge theory and orbicurves. In \autoref{sec:reduction}, following work of Langer-Simpson we reduce to the case of studying quasi-privileged local systems which projectively factor through an orbicurve by showing that the local systems in Case $1$ and Case $2$ are of geometric origin. Then in \autoref{sec:rigidity} we show that these quasi-privileged local systems are cohomologically rigid, and therefore by \autoref{thm:rank 3 rigid} they are of geometric origin. Finally in \autoref{sec:examples} we give an example of a privileged but not rigid local system.

\subsection{Acknowledgements}
We would like to thank H\'el\`ene Esnault for many helpful comments, both for mathematical correctness and exposition. Needless to say, this project would not be possible without her help. We also thank her for discussing the ideas in \cite{Esnault_Groechenig} with us. We would also like to thank the organizers of the workshop "A Sprint on Local Systems" in Wuppertal, Germany. The authors worked on Project H for this workshop (mentored by H\'el\`ene Esnault), which resulted in this paper. The first author thanks Kang Zuo for helpful discussions and suggestions. The second author thanks Daniel Litt for many helpful conversations. Part of this project was completed while the second author was at the Max Planck Institute for Mathematics in Bonn, Germany, and he thanks the MPIM for their hospitality.

\subsection{AI Disclosure}
ChatGPT as well as access to AI models provided to the second author by Surge AI and Google DeepMind were used to check the consistency of arguments, search for references, and improve the quality of exposition. The authors also obtained \autoref{prop:Hom vanishes} in discussion with ChatGPT, where ChatGPT suggested the idea of comparing the kernel $K$ appearing in the argument with the fundamental group of a fiber $F$.

\section{Background on Higgs bundles, local systems, and orbicurves}\label{sec:background}

We collect some facts about non-abelian Hodge theory, as well as separately some facts about orbicurves.

\subsection{Higgs bundles and local systems}
We briefly discuss non-abelian Hodge theory in the non-compact setting. This is the work of many people, including Hitchin, Corlette, and Simpson. For precise definitions, we point the reader to \cite[Section 2 and Section 4]{Landesman_Litt} for a thorough discussion. The version of non-abelian Hodge theory we will use is \autoref{thm:non-abelian hodge theory}, proven by Simpson \cite{Simpson_Harmonic} in the generality of smooth (not necessarily proper) curves, and also for smooth projective varieties in \cite{Simpson}.

\begin{theorem}\label{thm:non-abelian hodge theory}
\begin{enumerate}
\item Let $X$ be a smooth proper curve over the complex numbers and $D$ a reduced effective divisor on $X$ (i.e., a finite collection of points). There is a real-analytic isomorphism of moduli spaces between 
\begin{align*}
M_{\betti}(X\setminus D, r) \cong M_{\higgs}(X,D, r)
\end{align*}
where $M_{\betti}(X\setminus D, r)$ is the coarse moduli space of semi-simple rank $r$ local systems on $X \setminus D$, and $M_{\higgs}(X,D, r)$ is the coarse moduli space of polystable rank $r$ parabolic Higgs bundles on $(X,D)$ of parabolic degree zero. Under this isomorphism, irreducible local systems on $X \setminus D$ correspond to stable parabolic Higgs bundles on $(X,D)$.

The local systems underlying pure polarizable complex variations of Hodge structure correspond to the parabolic Higgs bundles $(E_\star, \theta)$ such that for all $t \in \bg_m$, $(E_\star, t \theta) \cong (E_\star, \theta)$. These Higgs bundles are precisely those which are of the form $(E_\star, \theta) \cong (\oplus E^p_\star, \theta)$ with $\theta(E^p) \subseteq E^{p-1} \otimes \Omega_X^1(\log D)$.

\item When $X$ is a smooth projective variety over the complex numbers, there is a real-analytic isomorphism between $M_{\betti}(X, r)$ and $M_{\higgs}(X, r)$ where $M_{\betti}(X, r)$ is the coarse moduli space of rank $r$ semi-simple local systems on $X$ and $M_{\higgs}(X, r)$ is the coarse moduli space of rank $r$ polystable Higgs bundles on $X$ with vanishing rational Chern classes. Under this isomorphism, irreducible local systems on $X$ correspond to stable Higgs bundles on $X$.

The local systems underlying pure polarizable complex variations of Hodge structure correspond to the Higgs bundles $(E, \theta)$ such that for all $t \in \bg_m$, $(E, t\theta) \cong (E, \theta)$. These Higgs bundles are precisely those which are of the form $(E, \theta) \cong (\oplus E^p, \theta)$ with $\theta(E^p) \subseteq E^{p-1} \otimes \Omega_X^1$.
\end{enumerate}
\begin{proof}
See \cite[Theorem on page 718]{Simpson_Harmonic} for the correspondence in the setting of non-proper curves, and \cite[Theorem 8]{Simpson_Harmonic} for the description of complex variations of Hodge structure. For the same statements for smooth projective varieties, see \cite[Theorem 1 and Lemma 4.1]{Simpson}.
\end{proof}
\end{theorem}

For our purposes, in the setting of non-proper curves, the parabolic structure at the punctures is not important and one can treat these objects as they are in the projective setting.

\subsection{Orbicurves}

We recall the definition of orbicurves as well as the notion of a local system factoring through an orbicurve. We follow the convention of \cite[Section 2 and Section 3]{Langer_Simpson}.

\begin{definition}
A \emph{orbicurve} $C$ is a smooth one-dimensional Deligne-Mumford stack over $\bc$ such that its generic stabilizer group is trivial.
\end{definition}

Let $C$ be a smooth projective curve over the complex numbers, $X$ a smooth projective complex variety, and $f:X \to C$ a fibration (i.e., a surjective morphism of varieties which we assume has connected fibers). Given a closed point $c \in C$, we write $X_c = f^{-1}(c)$. As a divisor in $X$, we may write $X_c = \sum a_i D_i$ and we call the greatest common divisor of the coefficients $a_i$ the multiplicity of the fiber. Let $p_1,\dots, p_k$ be the points at which the multiplicity of the fiber is greater than one. Let $c \in C$ be a point where the multiplicity of the fiber is one. Such a closed point exists as we are over a field of characteristic zero.

\begin{definition}
We define the \emph{orbifold fundamental group} of the fibration $f: X \to C$ to be the group
\begin{align*}
\pi_1^{\orb}(C_f, c) := \pi_1(C \setminus \{p_1,\dots, p_k\}, c) / N
\end{align*}
where $N$ is the normal subgroup of $\pi_1(C \setminus \{p_1,\dots, p_k\}, c)$ generated by $\gamma_1^{n_1}, \dots, \gamma_k^{n_k}$ where $\gamma_i$ is the homotopy class of a loop around the puncture $p_i$ and $n_i$ is the multiplicity of the fiber over $p_i$.
\end{definition}

\begin{remark}
Let $C^{\coarse}$ be the coarse space of the orbicurve $C$. Then, $C$ is isomorphic to the root stack $C^{\coarse}[p_1/n_1, \dots, p_k/n_k]$ and the orbifold fundamental group of $C$ is isomorphic to $\pi_1^{\orb}(C_f)$, where $f: X \to C^{\coarse}$ is a fibration such that the multiplicity of the fibers is $n_i$ at $p_i$. Therefore, when we are in the presence of a fibration $f: X \to C$ we will often write $C$ to also mean $C^{\coarse}[p_1/n_1, \dots, p_k/n_k]$.
\end{remark}

Now, let $X$ be a smooth projective variety and $C$ a curve. Observe that projectivity of $X$ implies that any fibration $f: X \to C$ must be projective.

\begin{definition}
Let $\bv$ be a rank $r$ local system on $X$ and let $x \in X$ be a closed point. We say that $\bv$ \emph{projectively factors through an orbicurve} if there is a curve $C$, a fibration $f: X \to C$, and a $\pgl_r$-representation of $\pi_1^{\orb}(C_f,f(x))$ such that the following diagram commutes
\begin{center}
\begin{tikzcd}
	{\pi_1(X,x)} & {\gl_r(\bc)} \\
	{\pi_1^{\orb}(C_f, f(x))} & {\pgl_r(\bc)}
	\arrow[from=1-1, to=1-2]
	\arrow[from=1-1, to=2-1]
	\arrow[from=1-2, to=2-2]
	\arrow[from=2-1, to=2-2]
\end{tikzcd}.
\end{center}
\end{definition}

The main tool we will use is the homotopy exact sequence.
\begin{theorem}\label{thm:homotopy exact sequence}
Let $f: X \to C$ be a fibration, and let $c \in C \setminus \{p_1,\dots, p_k\}$. Let $X_c = f^{-1}(c)$ and let $x \in X_c$ be a point. If $c \in C$ is chosen generically, or $f$ is proper, then we have an exact sequence
\begin{align*}
\pi_1(X_c, x) \to \pi_1(X,x) \to \pi_1^{\orb}(C_f, c) \to 1.
\end{align*}
\begin{proof}
See \cite[Theorem 2.1]{Langer_Simpson}.
\end{proof}
\end{theorem}

Orbicurves also have a notion of Euler characteristic. Let $C$ be a smooth projective curve of genus $g$, and let $f: X \to C$ be a fibration. Write $p_1,\dots, p_k$ to be the points in $C$ whose fibers have multiplicity $n_1,\dots, n_k$ which are larger than one.
\begin{definition}

The \emph{orbifold Euler characteristic} relative to $f$ is
\begin{align*}
\chi_{\orb}(C_f) &= 2 - 2g - \sum_{i = 1}^k \left (1 - \frac{1}{n_i} \right ).
\end{align*}
\end{definition}

We use Poincar\'e duality later on.
\begin{proposition}\label{prop:Poincare duality}
Suppose $\pi_1^{\orb}(C_f)$ is not virtually abelian and let $\rho: \pi_1^{\orb}(C_f) \to \gl(V)$ be some representation. Then, there is a perfect pairing $H^i(\pi_1^{\orb}(C_f), \rho) \times H^{2 - i}(\pi_1^{\orb}(C_f), \rho^\vee) \to \bc$. 
\begin{proof}
The virtually abelian case is the case where $\chi_{\orb}(C_f) \geq 0$; see \cite[Proposition 5.4 and Proposition 5.5]{Behrend_Noohi}. Therefore, $\chi_{\orb}(C_f) < 0$ since $\pi_1^{\orb}(C_f)$ is not virtually abelian. Therefore, $C$ is aspherical by \cite[Proposition 5.3]{Behrend_Noohi} and \cite[Corollary 5.10]{Behrend_Noohi}. Hence $C$ is very good in the sense of Porti \cite[First paragraph of Section 3]{Porti}. Then, we apply \cite[Proposition 3.4]{Porti}. \cite[Proposition 3.4]{Porti} is stated for $\ad(\rho)$, but also holds for general coefficients (see \cite[First paragraph of Section 3]{Porti}). We use that $\pi_1^{\orb}(C_f)$ is \emph{not} virtually abelian to guarantee that $C$ is aspherical.
\end{proof}
\end{proposition}

\section{Reduction to factoring through an orbicurve}\label{sec:reduction}

We recall some facts about local systems here. Much of the content in this section is spelled out in great detail in \cite{Langer_Simpson}. We first discuss some general properties about quasi-privileged local systems, and then we will restrict to studying quasi-privileged local systems with trivial determinant. The key collection of properties we will use is outlined in \autoref{prop:privileged galois}. 

For a Higgs bundle $(E, \theta)$, we will write $H^i_{\higgs}((E,\theta))$ to mean the hypercohomology of the complex
\begin{align*}
H^i_{\higgs}((E,\theta)) &:= \bh^i (E \xrightarrow{\theta} E \otimes \Omega_X^1[-1] \xrightarrow{\theta} E \otimes \Omega_X^2[-2] \to  \cdots).
\end{align*}
If $(E,\theta)$ corresponds to the local system $\bl$, then $H^i_{\higgs}((E,\theta)) \cong H^i(X, \bl)$.

\begin{proposition}\label{prop:privileged galois}
Let $X$ be a smooth projective variety over the complex numbers, and let $\bv$ be a quasi-privileged local system in $M_{\betti}(X, r, \bl)$. Then, the following properties hold
\begin{enumerate}
\item $\bv$ underlies a pure polarizable complex variation of Hodge structure,
\item $\bv$ is defined over the ring of integers of a number field.
\item for all $\sigma \in \aut(\bc)$, the Galois conjugate $\sigma \cdot \bv$ of $\bv$ is quasi-privileged in $M_{\betti}(X,r, \sigma \cdot \bl)$.
\end{enumerate}
\begin{proof}
For (3), the same proof as \cite[Proposition 2.9]{Esnault_Groechenig} works.

For (1), we use non-abelian Hodge theory \autoref{thm:non-abelian hodge theory}, which says there is a real-analytic isomorphism from $M_{\betti}(X,r, \bl)$ to $M_{\higgs}(X, r, (L,0))$, the moduli space of rank $r$ Higgs bundles whose determinant bundle is isomorphic to the rank one Higgs bundle $(L,0)$ corresponding to the torsion rank one local system $\bl$. Then, $M_{\higgs}(X,r,(L,0))$ admits a $\bg_m$-action given by scaling the Higgs field. The fixed points of this $\bg_m$-action correspond to the complex variations of Hodge structure \cite[Lemma 4.1]{Simpson}. For a Higgs bundle $(E,\theta)$, the $\bg_m$-action induces an isomorphism of cohomology groups $H_{\higgs}^1(X, \End^0((E,\theta))) \cong H_{\higgs}^1(X,\End^0((E,t \theta)))$. Therefore, the $\bg_m$-action preserves the locus of Higgs bundles corresponding to quasi-privileged local systems. But the orbit of the $\bg_m$-action is connected, so the orbit of a Higgs bundle corresponding to a quasi-privileged local system is connected and discrete. This implies that the orbit is a single point, implying that quasi-privileged local systems underlie complex variations of Hodge structure.

Then, for (2) we let $Z_d := \{ \bw \in M_{\betti}(X,r, \bl) : h^1(X, \End^{0}(\bw)) \geq  d \}$. Observe that $M_{\betti}(X,r, \bl)$ is defined over a number field $K$, so $Z_d$ is as well. This is because the defining equations for $M_{\betti}(X,r, \bl)$ have coefficients in $K$, so the Jacobian of these defining equations has coefficients lying in $K$ as well. We call the Jacobian $J(x)$, and we let it be a matrix of size $\ell_1 \times \ell_2$ with $\ell_2$ the dimension of the Betti moduli space. Then, $Z_d$ is simply the locus of points in $M_{\betti}(X,r, \bl)$ for which the $J(x)$ has rank at most $\ell_2 - d$. This is a closed condition, so $Z_d$ is a closed subvariety of $M_{\betti}(X, r, \bl)$ over $K$.

By upper-semicontinuity of the Zariski tangent space, we have that $Z_d$ is closed in $M_{\betti}(X,r, \bl)$ for all $d$. Therefore, $Z_d \setminus Z_{d + 1}$ is locally closed, and consists of the local systems with Zariski tangent space of dimension $d$. Taking $d = \dim H^1(X, \End^{0}(\bv))$, we have that $Z_d \setminus Z_{d + 1}$ is a finite set of points. Since $Z_d$ and $Z_{d+1}$ are defined over number fields, so are all of the points in $Z_d \setminus Z_{d + 1}$. Then, the methods in \cite{Esnault_Groechenig_Integrality} apply without change and imply that the quasi-privileged local systems are actually integral.
\end{proof}
\end{proposition}

\subsection{Weight one $\bz$-variations of Hodge structure}
We recall some results on weight one variations of Hodge structure. By this, we mean the Hodge types are contained in $\{(0,1), (1,0)\}$. Note that in our setting, a rank $3$ irreducible local system underlying a variation of Hodge structure can always be put into weight one unless the Hodge numbers of the variation of Hodge structure are $(1,1,1)$. This is because the Hodge numbers of an irreducible rank $3$ variation of Hodge structure can be $(3)$, $(2,1)$, $(1,2)$, or $(1,1,1)$, and in the first three cases they are weight one (after possibly shifting the Hodge types).

\begin{proposition}\label{prop:weight one vhs}
Let $\bv$ be a local system defined over a number field $K$ on a smooth projective complex variety $X$, and suppose that all Galois conjugates of $\bv$ underlie a pure polarized complex variation of Hodge structure of weight one. Furthermore, suppose that all Galois conjugates of $\bv$ are integral. Then, $\bv$ is of geometric origin. In fact, $\bv$ arises as a summand of the cohomology of an abelian scheme over $X$.
\begin{proof}
Let $\bv$ be such a local system and for any $\sigma \in \gal(K/\bq)$ let $\sigma \cdot \bv$ be the Galois conjugate of $\bv$ by $\sigma$. Consider the local system which we obtain by summing over all embeddings of $K$ into $\bc$:
\begin{align*}
\bl = \bigoplus_{\sigma K \hookrightarrow \bc} \sigma \cdot \bv.
\end{align*}
By construction, $\bl$ is a local system underlying a pure polarized complex variation of Hodge structure of weight one, and furthermore it is defined over $\bz$ because it is fixed under Galois conjugation and all Galois conjugates $\sigma \cdot \bv$ are integral. The fact that $\bl$ is a polarizable weight one integral variation of Hodge structure is due to \cite[Proposition 7.1 and Lemma 7.2]{Langer_Simpson}. Therefore, it is enough to show that $\bl$ is of geometric origin. But this is a standard fact for weight one pure polarized $\bz$-variations of Hodge structure. See \cite[Rappel 4.4.3]{Deligne_Hodge_2}, where they show that $\bl$ comes from an abelian scheme over $X$.
\end{proof}
\end{proposition}

Therefore, combining \autoref{prop:weight one vhs} and \autoref{prop:privileged galois} allows us to conclude the following:
\begin{proposition}\label{prop: privileged (1,1,1)}
Let $\bv$ be a quasi-privileged local system of rank $3$ on a smooth projective complex variety $X$, and none of the Galois conjugates of $\bv$ underlie pure polarized complex variations of Hodge structure with Hodge numbers which are $(1,1,1)$. Then, $\bv$ is of geometric origin.
\end{proposition}

Therefore, \autoref{prop: privileged (1,1,1)} takes care of Case $1$ outlined in \autoref{subsec:methods}.

\subsection{Hodge type $(1,1,1)$ with non-Zariski-dense monodromy}

\autoref{prop: privileged (1,1,1)} allows us to restrict to the case where $\bv$ is a quasi-privileged local system underlying a pure polarized complex variation of Hodge structure with Hodge numbers $(1,1,1)$. Finally, we leverage the following theorem due to Langer-Simpson.

\begin{theorem}\label{thm: reduction to curve}
Let $\bv$ be a rank $3$ irreducible local system with trivial determinant coming from a pure polarized complex variation of Hodge structure with Hodge numbers $(1,1,1)$. Then, 
\begin{enumerate}
\item if the monodromy group of $\bv$ is Zariski-dense in $\sl_3(\bc)$, the local system projectively factors through an orbicurve,
\item if the monodromy group of $\bv$ is not Zariski-dense in $\sl_3(\bc)$, there is some smooth projective variety $Y$, a finite cover $g: Y \to X$, and a rank $2$ local system $\bw$ on $Y$ such that $g^* \bv = \sym^2 \bw$.
\end{enumerate}
\begin{proof}
This is \cite[Theorem 5.1]{Langer_Simpson} along with the observation that a connected reductive subgroup of $\sl_3$ which acts irreducibly on $\bc^3$ is given by $\sl_3$ or $\so(3)$ (which is isogeneous to $\sl_2$ via the map $\sym^2$). In case (2), Langer-Simpson show that there is some rank $2$ local system $\bw'$ on $Y$ and a rank one local system $\chi$ on $X$ so that $g^* \bv \cong g^* \chi \otimes \sym^2 \bw'$ where $\chi^{\otimes 3} = 1$. But then, taking $\bw = g^* \chi^{-1} \otimes \bw'$ yields $\sym^2 \bw \cong g^* (\chi^{\otimes -2} ) \otimes \sym^2 \bw' \cong g^* \chi \otimes \sym^2 \bw' \cong g^* \bv$.
\end{proof}
\end{theorem}

In the non-Zariski dense case, the finite cover $g: Y \to X$ provided by \autoref{thm: reduction to curve}(2) admits the unit map $\bv \to g_* g^* \bv$. After shrinking to a Zariski-dense open set $U \subseteq X$, $g: g^{-1}(U) \to U$ is \'etale and so $g_* g^* \bv$ is a local system of geometric origin on $U$. Therefore, if $g^* \bv$ is of geometric origin then so is $\bv$. We now show that $g^* \bv$ is of geometric origin. The main ideas in the following argument are in \cite[Corollary 3.5 and Corollary 3.7]{Viehweg_Zuo}, but we include a proof for completeness.

\begin{proposition}\label{prop:sym is motivic}
Suppose $\bv$ is a quasi-privileged local system on a smooth projective variety $X$, and suppose that there is a smooth projective variety $Y$ and a finite cover $g: Y \to X$ such that $g^* \bv \cong \sym^2 \bw$ where $\bw$ is a rank $2$ local system on $Y$. Then, $\bw$ (and therefore $\bv$) is of geometric origin.
\begin{proof}
Let $\bv$ be quasi-privileged on $X$. Then, $\bv$ underlies a pure polarized variation of Hodge structure and is defined over the ring of integers of a number field by \autoref{prop:privileged galois}. Then for all $\sigma \in \aut(\bc)$, we have that $g^* (\sigma \cdot \bv) \cong \sym^2 (\sigma \cdot \bw)$.

Since $\bw$ is a rank $2$ local system, if it (and all of its Galois conjugates $\sigma \cdot \bw$) underlie pure polarized complex variations of Hodge structure and are defined over rings of integers of number fields, then $\bw$ is a summand of the cohomology of an abelian scheme over $Y$ due to \autoref{prop:weight one vhs}. This is because the Hodge numbers of an irreducible rank $2$ variation of Hodge structure are $(2)$ or $(1,1)$ and therefore we may always move the Hodge numbers of $\sigma \cdot \bw$ to lie in $\{(1,0), (0,1)\}$. In the semi-simple case, it will be a direct sum of rank one variations of Hodge structure and therefore we may always shift indices to force both summands to have Hodge numbers lying in $\{(1,0),(0,1)\}$.

To prove that $\bw$ underlies a variation of Hodge structure and is integral, we first establish that $\bw$ is semi-simple. We use semi-simplicity to check both integrality and that it underlies a variation of Hodge structure. If $\bw$ is irreducible, then there is nothing to prove. So we assume that there is a rank one sub-local system $\chi \subseteq \bw$. Then we have an inclusion of local systems $\chi \otimes \bw \hookrightarrow \sym^2 \bw$ via the map $v \otimes w \mapsto (v \otimes w + w \otimes v)/2$. But since $\sym^2 \bw$ is a semi-simple local system, $\chi \otimes \bw$ is as well. This implies that $\bw$ itself is semi-simple.

\textbf{Variation of Hodge Structure:} 
Since $\bw$ is semi-simple, $\bw$ underlies a pure polarized complex variation of Hodge structure if and only if the corresponding Higgs bundle $(E,\theta)$ under non-abelian Hodge theory (\autoref{thm:non-abelian hodge theory}) is $\bg_m$-fixed (see \cite[Corollary 4.2]{Simpson}). Let $\rho: \pi_1(Y) \to \gl_2(\bc)$ be the representation corresponding to $\bw$, and $\rho_t: \pi_1(Y) \to \gl_2(\bc)$ be the representation corresponding to $(E, t \theta)$. We claim that $\rho \cong \rho_t$, implying that $(E,\theta)$ is $\bg_m$-fixed. 

We first note that $\sym^2 \bw$ is a sub-quotient of $\bw \otimes \bw$, and the $\bg_m$-action and the tensor product commute. As a consequence, $(\sym^2 \rho_t) \cong (\sym^2 \rho)_t$. Indeed,
\begin{align*}
(t \cdot (E ,\theta)) \otimes (t \cdot (E,\theta)) &\cong  (E, t\theta) \otimes (E, t\theta) \cong (E \otimes E, (t\theta) \otimes 1 + 1 \otimes (t\theta)) \\
& \cong  (E \otimes E, t(\theta \otimes 1 + 1 \otimes \theta)) \cong  t \cdot (E \otimes E, \theta \otimes 1 + 1 \otimes \theta).
\end{align*}

Now to check that $\rho \cong \rho_t$, we use the following criterion for semi-simple representations: two semi-simple representations $\rho$ and $\rho_t$ are isomorphic if and only if for all $\gamma \in \pi_1(Y)$, $\tr \rho_t(\gamma) = \tr \rho(\gamma)$. To check this, we fix some $\gamma \in \pi_1(Y)$. Then $\sym^2 \rho \cong \sym^2 \rho_t$ since $\sym^2 \bw$ underlies a variation of Hodge structure. Therefore, we have that 
\begin{align*}
\det(\rho_t(\gamma))^3 = \det(\sym^2 \rho_t(\gamma)) = \det(\sym^2 \rho(\gamma)) = \det(\rho(\gamma))^3.
\end{align*}
Therefore $\det(\rho_t(\gamma))$ is a cube-root of $\det(\rho(\gamma))^3$ for all $t \in \bg_m$. But the cube roots of $\det(\rho(\gamma))^3$ form a discrete set and $\bg_m$ is connected, so $\det(\rho_t(\gamma)) = \det (\rho(\gamma))$. Similarly, we have that
\begin{align*}
\tr(\rho_t(\gamma))^2 = \tr(\sym^2 \rho_t(\gamma)) + \det \rho_t(\gamma) = \tr(\sym^2 \rho(\gamma)) + \det \rho(\gamma) = \tr(\rho(\gamma))^2.
\end{align*}
Again, the square-roots of $\tr(\rho(\gamma))^2$ form a discrete set and $\bg_m$ is connected so $\tr \rho_t(\gamma) = \tr \rho(\gamma)$. Therefore, $\rho \cong \rho_t$ which implies that $\bw$ underlies a pure polarizable complex variation of Hodge structure.

\textbf{Integrality:} 
Now let $K$ be the field of definition of $\sym^2 \bw$. To check integrality of $\bw$, we use the following integrality criterion due to Bass \cite[Corollary 2.3 and Corollary 2.5]{Bass}, specialized to our situation: if $\rho: \pi_1(Y) \to \gl_2(\bc)$ is a semi-simple representation, then there is a number field $L$ and a matrix $M \in \gl_2(\bc)$ such that $M \rho M^{-1}$ has image lying in $\gl_2(\mathcal{O}_L)$ if and only if for all $\gamma \in \pi_1(Y)$ we have that $\tr(\rho(\gamma))$ is an algebraic integer. We note that Bass proves it only in the irreducible setting, but the same statement can be extended to the semi-simple case by \cite[Remark 2]{deSmit}.

But checking that $\tr(\rho(\gamma))$ is an algebraic integer is easy. The number $\tr(\rho(\gamma))$ satisfies the relation already used above involving traces, which is a monic polynomial with coefficients lying in $\overline{\bz}$. Therefore, $\tr(\rho(\gamma))$ is indeed an algebraic integer and so Bass's criterion implies that $\rho$ is defined over the ring of integers of some number field $L$.
\end{proof}
\end{proposition}

This completes Case $2$ outlined in \autoref{subsec:methods}.

\section{Cohomological rigidity when factoring through an orbicurve}\label{sec:rigidity}

By \autoref{prop: privileged (1,1,1)}, \autoref{prop:sym is motivic}, and \autoref{thm: reduction to curve} we have reduced the proof of \autoref{thm:quasi-privileged implies motivic} to showing that motivicity holds for rank $3$ quasi-privileged local systems underlying variations of Hodge structure with Hodge numbers $(1,1,1)$, with Zariski-dense monodromy in $\sl_3(\bc)$, and projectively factor through an orbicurve. We show something stronger, namely that in fact they are cohomologically rigid. This will complete Case $3$ of \autoref{subsec:methods}.

\begin{theorem}\label{thm:quasi-priviliged implies rigid}
Let $X$ be a smooth projective variety over $\bc$, and let $\bv$ be an irreducible rank $3$ quasi-privileged local system with trivial determinant underlying a pure polarized variation of Hodge structure with Hodge numbers $(1,1,1)$. Suppose $\bv$ has Zariski-dense monodromy in $\sl_3(\bc)$ (and hence projectively factors through an orbicurve). Then, $\bv$ is cohomologically rigid.
\end{theorem}
Then, \autoref{thm:quasi-priviliged implies rigid} implies \autoref{thm:quasi-privileged implies motivic} by applying Langer-Simpson's results on rigid rank $3$ local systems (\cite[Theorem 1.3]{Langer_Simpson}) along with the integrality results of Esnault and Groechenig (\cite[Theorem 1.1]{Esnault_Groechenig_Integrality}).

The content of this section is to prove \autoref{thm:quasi-priviliged implies rigid}. We first prove some helpful lemmas along the way.

\begin{remark}
In fact, we can already conclude \autoref{thm:quasi-privileged implies motivic} if we are working with rank $3$ privileged local systems on smooth projective varieties with trivial determinant. This is because in this case, a privileged local system is necessarily self-dual. And a rank $3$ self-dual local system has monodromy group contained in $\mathrm{SO}_3$ which is isogenous to $\sl_2$.
\end{remark}

\begin{lemma}\label{lem:weight one}
Let $S$ be a smooth proper curve over the complex numbers and let $D$ be a reduced effective divisor. Let $f :X \to S \setminus D$ be a smooth proper morphism. Let $(E_\star, \theta)$ be the graded parabolic Higgs bundle corresponding to the local system $R^1 f_* \bc$. Then, every sub-Higgs bundle of $(E_\star, \theta)$ of parabolic degree zero has Higgs field which is nilpotent of order at most two, and every sub-Higgs bundle of $(E_\star, \theta)$ has at most two graded pieces.
\begin{proof}
Since $R^1 f_* \bc$ is semi-simple (see \cite[Th\'eor\`eme 4.2.6]{Deligne_Hodge_2}, it suffices to prove this statement for the entire Higgs bundle $(E_\star,\theta)$. The local system $R^1 f_* \bc$ underlies the variation of Hodge structure $(V, \nabla, F^\bullet ) = (R^1 f_* \bc \otimes \mathscr{O}_{S \setminus D}, \id \otimes d, F^\bullet)$ where the Hodge filtration $F^\bullet$ is given by
\begin{align*}
F^0 &= R^1 f_* \bc \otimes \mathscr{O}_{S \setminus D}\\
F^1 &= R^0 f_* \Omega^1_{X/(S \setminus D)}\\
F^2 &= 0
\end{align*}
and the parabolic Higgs bundle $(E_\star, \theta)$ is given by the associated gradeds of the Deligne canonical extension of $(V,\nabla)$ over $S \setminus D$ to $(\overline{V}, \overline{\nabla})$ over $S$ (see \cite[Remarques 5.5(i)]{Deligne_canonical_extension}) with respect to the canonical Griffiths transverse filtration $\overline{F^\bullet}$ on $(\overline{V}, \overline{\nabla})$ extending the Hodge filtration $F^\bullet$ on $(V,\nabla)$ (see \cite[Section 7]{Brunebarbe}).

That is, the parabolic Higgs bundle has underlying bundle 
\begin{align*}
E \cong \overline{F^1}/\overline{F^2} \oplus \overline{F^0}/\overline{F^1}
\end{align*}
and the Higgs field is given by 
\begin{align*}
\theta:  \overline{F^1}/\overline{F^2} \to  \overline{F^0}/\overline{F^1} \otimes \Omega_S^1(\log D).
\end{align*}
In other words, the graded Higgs bundle $(E_\star, \theta)$ has two graded pieces, and the nilpotence order of $\theta$ is bounded above by the number of graded pieces.
\end{proof}
\end{lemma}

Let $\bv$ be an irreducible rank $3$ local system underlying a variation of Hodge structure on $X$, and suppose it (projectively) factors through an orbicurve by a fibration $f : X \to C$. Let $\bp(\bv)$ be the projectivization of $\bv$, and $\rho: \pi_1^{\orb}(C_f) \to \pgl_3(\bc)$ the corresponding group homomorphism obtained from $\bv$ projectively factoring through $C$. We claim that if $\bv$ has Zariski-dense monodromy in $\sl_3(\bc)$, then $\pi_1^{\orb}(C_f)$ \emph{cannot} be virtually abelian.

Suppose $\pi_1^{\orb}(C_f)$ were virtually abelian. Then by definition, there is some finite index subgroup $\Gamma \subseteq \pi_1^{\orb}(C_f)$ such that $\Gamma$ is abelian. Since $\bv$ had Zariski-dense monodromy and $\gl_3(\bc) \to \pgl_3(\bc)$ is a quotient by $\bg_m$, $\rho$ also has Zariski-dense image in $\pgl_3(\bc)$. Then, restricting $\rho$ to $\Gamma$ also has Zariski-dense image. To see this, let $H \subseteq \pgl_3(\bc)$ be the Zariski-closure of $\rho(\Gamma)$. We pick coset representatives $\gamma_1,\dots, \gamma_n$ for $\Gamma$ so that $\gamma_1 \Gamma \cup \dots \cup \gamma_n \Gamma = \pi_1^{\orb}(C_f)$. Then, since the image of $\rho$ is Zariski-dense,
\begin{align*}
\rho(\gamma_1) H \cup \dots \cup \rho(\gamma_n) H = \pgl_3(\bc).
\end{align*}
Since $\pgl_3(\bc)$ is irreducible (as a scheme), there is some $\rho(\gamma_i) H$ which is all of $\pgl_3(\bc)$. This implies that $\rho(\Gamma)$ is Zariski-dense in $\pgl_3(\bc)$. Finally since $\Gamma$ is abelian, $\rho(\Gamma)$ is abelian. But the Zariski closure of an abelian group is still abelian, which would imply that $\pgl_3(\bc)$ is abelian. This is not the case, so we have a contradiction. Therefore, $\pi_1^{\orb}(C_f)$ is not virtually abelian.

Now, recall that the Lie algebras of $\sl_3(\bc)$ and $\pgl_3(\bc)$ are isomorphic. Then, we have a commutative diagram of group homomorphisms
\begin{center}
\begin{tikzcd}
	{\pi_1(X)} & {\sl_3(\bc)} & \\
	{\pi_1^{\orb}(C_f)} & {\pgl_3(\bc)} & {\gl(\mathfrak{sl}_3(\bc))}
	\arrow[from=1-1, to=1-2]
	\arrow["f"', from=1-1, to=2-1]
	\arrow[from=1-2, to=2-2]
	\arrow["{\ad_{\sl_3} = \End^0}", from=1-2, to=2-3]
	\arrow["\rho"', from=2-1, to=2-2]
	\arrow["{\ad_{\pgl_3}}"', from=2-2, to=2-3]
\end{tikzcd}
\end{center}
where $\ad_{\sl_3}$ and $\ad_{\pgl_3}$ are the adjoint representations of $\sl_3(\bc)$ and $\pgl_3(\bc)$, respectively. By the commutativity of the diagram, we see that pulling back $\ad_{\pgl_3} ( \rho)$ along the map $f$ gives us a local system of $\bc$-vector spaces on $X$ which is isomorphic to $\End^{0}(\bv)$.

\subsection{The argument of the main theorem via the homotopy exact sequence}
We consider the exact sequence
\begin{align*}
1 \to K \to \pi_1(X) \to \pi_1^{\orb}(C_f) \to 1
\end{align*}
where $K$ is the kernel of $\pi_1(X) \to \pi_1^{\orb}(C_f)$. Observe that $K$ acts trivially on $\End^0(\bv)$.

\begin{remark}\label{rem:conjugation action}
There is an action of $\pi_1^{\orb}(C_f)$ on $H^1(K, \bc)$ given by "conjugation" which we describe. This is standard, but we give an explicit description here as we will make use of this action later on. Given a group $G$, we will write $[g]$ to mean the equivalence class of $g$ in the abelianization $G^{\ab} = G/[G,G]$.

Given some $\gamma \in \pi_1^{\orb}(C_f)$, we pick a lift $\widetilde{\gamma}$. Note that all other lifts of $\gamma$ are given by $u \cdot \widetilde{\gamma}$ for some $u \in K$. Now, let $[k] \in K^{\ab}$ and $k \in K$ a representative of $[k]$. We define the action of $\gamma$ on $[k]$ to be given by $[\widetilde{\gamma} k \widetilde{\gamma}^{-1}]$. We claim that this is a well-defined action. First, we observe that if we pick another lift $u \cdot \widetilde{\gamma}$ for some $u \in K$, we have that
\begin{align*}
[u \cdot \widetilde{\gamma} k (u \cdot \widetilde{\gamma})^{-1}] = [u \cdot \widetilde{\gamma} k \widetilde{\gamma}^{-1} \cdot u^{-1}] = [\widetilde{\gamma} k \widetilde{\gamma}^{-1}]
\end{align*}
where the last equality uses the fact that both $u$ and $\widetilde{\gamma} k \widetilde{\gamma}^{-1} \in K$, so in $K^{\ab}$ they commute. To check that the action is independent of the choice of representative of $[k]$, we let $k'$ be another representative of $[k]$. Then, $k^{-1} k'$ is a finite product of commutators in $K$. It is therefore enough to prove that every element of the form $aba^{-1}b^{-1}$, for some $a,b \in K$, is sent to $[K,K]$. One can check that 
\begin{align*}
[\widetilde{\gamma} (aba^{-1}b^{-1})\widetilde{\gamma}^{-1} ] = [\widetilde{\gamma} a \widetilde{\gamma}^{-1} \widetilde{\gamma} b \widetilde{\gamma}^{-1} (\widetilde{\gamma} a \widetilde{\gamma}^{-1})^{-1} (\widetilde{\gamma} b \widetilde{\gamma}^{-1})^{-1}].
\end{align*}
Again using that $\widetilde{\gamma} a \widetilde{\gamma}^{-1}$ and $\widetilde{\gamma} b \widetilde{\gamma}^{-1}$ are elements of $K$, we conclude that in $K^{\ab}$ the action does not depend on the choice of representatives. We obtain a well-defined action of $\pi_1^{\orb}(C_f)$ on $H^1(K, \bc)$.
\end{remark}

The Hochschild-Serre spectral sequence for $\End^{0}(\bv)$ yields a five-term exact sequence
\begin{align*}
0 \to H^1( \pi_1^{\orb}(C_f) , \ad_{\pgl_3} ( \rho)) \to H^1(X, \End^{0}(\bv)) \to H^1(K, \End^{0}(\bv))^{\pi_1^{\orb}(C_f)} \to H^2( \pi_1^{\orb}(C_f) ,\ad_{\pgl_3} ( \rho )) \to \dots
\end{align*}

Now, we note that $H^0(\pi_1^{\orb}(C_f) , \ad_{\pgl_3}(\rho))$ gives us the $\pi_1^{\orb}(C_f)$-invariant vectors in $\mathfrak{sl}_3(\bc)$ under the action of  $\ad_{\pgl_3}(\rho)$. But since $\rho$ is Zariski-dense and the center of $\pgl_3(\bc)$ is trivial, $H^0(\pi_1^{\orb}(C_f) , \ad_{\pgl_3}(\rho)) = 0$. Then by orbifold Poincar\'e duality (\autoref{prop:Poincare duality}) and self-duality of $\ad_{\pgl_3}(\rho)$, we have that $H^0(\pi_1^{\orb}(C_f) , \ad_{\pgl_3}(\rho))^\vee = H^2(\pi_1^{\orb}(C_f) , \ad_{\pgl_3}(\rho)) = 0$.

Therefore, we obtain the short exact sequence
\begin{align*}
0 \to H^1( \pi_1^{\orb}(C_f) , \ad_{\pgl_3} ( \rho)) \to H^1(X, \End^{0}(\bv)) \to H^1(K, \End^{0}(\bv))^{\pi_1^{\orb}(C_f)} \to 0.
\end{align*}

\begin{proposition}\label{prop:Hom vanishes}
In the above setting, $H^1( \pi_1^{\orb}(C_f) , \ad_{\pgl_3} ( \rho)) \cong H^1(X, \End^{0}(\bv))$.
\begin{proof}
Showing that $H^1( \pi_1^{\orb}(C_f) , \ad_{\pgl_3} ( \rho)) \cong H^1(X, \End^{0}(\bv))$ is the same as showing 
\begin{align*}
H^1(K, \End^{0}(\bv))^{\pi_1^{\orb}(C_f)} = 0.
\end{align*}
Since $K$ acts trivially on $\End^{0}(\bv)$, we have that
\begin{align*}
H^1(K, \End^{0}(\bv))^{\pi_1^{\orb}(C_f)} \cong \hom_{\pi_1^{\orb}(C_f)}(\End^{0}(\bv), H^1(K, \bc)).
\end{align*}

Now, write $C^\circ = C \setminus \{p_1,\dots, p_k, s_1,\dots, s_n\}$ where $p_1,\dots, p_k$ are the orbifold points in $C$ and $s_1,\dots, s_n$ are the points in $C$ with singular fibers. Note that $\{s_1,\dots, s_n\}$ and $\{p_1,\dots, p_k\}$ may have elements in common. Since $\pi_1^{\orb}(C_f)$ is a quotient of $\pi_1(C^\circ)$, we have a surjective map $\pi_1(C^\circ) \twoheadrightarrow \pi_1^{\orb}(C_f)$.

We pull everything back to $\pi_1(C^\circ)$, and surjectivity of the map $\pi_1(C^\circ) \twoheadrightarrow \pi_1^{\orb}(C_f)$ tells us that it is enough to prove that
\begin{align*}
\hom_{\pi_1(C^\circ)}(\End^{0}(\bv), H^1(K, \bc)) = 0.
\end{align*}

Now, let us restrict our fibration to $C^\circ$. Because we have removed the orbifold points and singular fibers, we have a smooth proper morphism $f: f^{-1}(C^\circ) \to C^\circ$. Let $F$ be a fiber of this fibration.

Then, the homotopy exact sequence yields an exact sequence
\begin{align*}
1 \to \pi_1(F) \to \pi_1(f^{-1}(C^\circ)) \to \pi_1(C^\circ) \to 1.
\end{align*}
Moreover, there is a surjective map $\pi_1(F) \to K$ by the orbifold homotopy exact sequence \autoref{thm:homotopy exact sequence}. Then, this gives us an injective morphism of complex vector spaces
\begin{align*}
H^1(K, \bc) = \hom(K, \bc) \hookrightarrow \hom(\pi_1(F), \bc) = H^1(F, \bc).
\end{align*}
The action of $\pi_1(C^\circ)$ on $H^1(K, \bc)$ factors through $\pi_1^{\orb}(C_f)$ and is described in \autoref{rem:conjugation action}. We claim that the group homomorphism
\begin{align*}
s:\pi_1(F)^{\ab} \to K^{\ab}
\end{align*}
is $\pi_1(C^\circ)$-equivariant. We have a commutative diagram
\begin{center}
\begin{tikzcd}
	1 & {\pi_1(F)} & {\pi_1(f^{-1}(C^\circ))} & {\pi_1(C^\circ)} & 1 \\
	1 & K & {\pi_1(X)} & {\pi_1^{\orb}(C_f)} & 1
	\arrow[from=1-1, to=1-2]
	\arrow[from=1-2, to=1-3]
	\arrow["q"', from=1-2, to=2-2]
	\arrow[from=1-3, to=1-4]
	\arrow["i", from=1-3, to=2-3]
	\arrow[from=1-4, to=1-5]
	\arrow["j", from=1-4, to=2-4]
	\arrow[from=2-1, to=2-2]
	\arrow[from=2-2, to=2-3]
	\arrow[from=2-3, to=2-4]
	\arrow[from=2-4, to=2-5]
\end{tikzcd}
\end{center}
For a group $G$, we will write $[g]$ to mean the equivalence class of $g \in G$ in the quotient $G^{\ab}$. Then given some $[\delta] \in \pi_1(F)^{\ab}$ with coset representative $\delta$ and $\gamma \in \pi_1(C^\circ)$ we have that
\begin{align*}
s(\gamma \cdot [\delta]) = s([\widetilde{\gamma} \delta \widetilde{\gamma}^{-1}]) = [q(\widetilde{\gamma} \delta \widetilde{\gamma}^{-1})] = [i(\widetilde{\gamma}) q(\delta) i(\widetilde{\gamma})^{-1}] = j(\gamma) \cdot [q(\delta)] = \gamma \cdot s([\delta]).
\end{align*}
Therefore, the map $s$ is indeed $\pi_1(C^\circ)$-equivariant. As a consequence, we have an injective morphism of $\pi_1(C^\circ)$-representations $H^1(K,\bc) \hookrightarrow H^1(F, \bc)$.

Therefore, it suffices to prove that 
\begin{align*}
\hom_{\pi_1(C^\circ)}(\End^{0}(\bv), H^1(F, \bc)) = 0.
\end{align*}

First, we observe that since $\bv$ is a local system with Zariski-dense monodromy and the Lie algebra of $\pgl_3$ is the simple Lie algebra $\mathfrak{sl}_3$, the adjoint local system $\End^{0}(\bv)$ is an irreducible local system. Furthermore, Deligne proved that local systems of geometric origin are semi-simple \cite[Th\'eor\`eme 4.2.6]{Deligne_Hodge_2} (and indeed, complex variations of Hodge structure in general are semi-simple). Therefore, it is enough to prove that no summand of $R^1 f_* \bc$ is isomorphic to $\End^{0}(\bv)$. We note that the restriction of $\End^0(\bv)$ to $f^{-1}(C^\circ)$ underlies a variation of Hodge structure, so $\End^0(\bv)$ gives us a variation of Hodge structure on $C^\circ$ as we can check if a local system underlies a variation of Hodge structure after pulling back along a surjection of fundamental groups (see \cite[Corollary 4.3]{Simpson} for the projective case, or \cite[Theorem 25.30]{Mochizuki} in general).

Now, let $\overline{C^\circ}$ be the compactification of $C^\circ$ with simple normal crossings divisor $D$. Since $C^\circ$ is a smooth curve, $D$ is just a collection of points. By \autoref{lem:weight one}, every summand of $R^1 f_* \bc$ corresponds to a parabolic Higgs bundle on the curve-divisor pair $(\overline{C^{\circ}},D)$ with Higgs field of nilpotence order at most two. In other words, every summand of $R^1 f_* \bc$ corresponds to a Higgs bundle with at most two graded pieces.

The local system $\bv$ on $X$ corresponds to a Higgs bundle $(E,\theta)$. Recall that $\bv$ underlies a complex variation of Hodge structure, and Hodge numbers of $\bv$ are $(1,1,1)$. This tells us that $(E, \theta)$ is a system of Hodge bundles of the form (up to a shift of indices)
\begin{align*}
E \cong E^3 \oplus E^2 \oplus E^1
\end{align*}
with Higgs field satisfying $\theta(E^p) \subseteq E^{p - 1} \otimes \Omega^1_{X}$. Now, the Higgs bundle $(T_\star, \omega)$ corresponding to $\End^{0}(\bv)$ on $C^\circ$ has $5$ graded pieces because $\End^0(\bv)$ on $X$ corresponds to a Higgs bundle $(S, \eta)$ with $5$ graded pieces
\begin{align*}
S^{2} &= \shom(E^1, E^3)\\
S^{1} &= \shom(E^1, E^2) \oplus \shom (E^2, E^3)\\
S^0 &= (\shom(E^1, E^1) \oplus \shom(E^2, E^2) \oplus \shom(E^3, E^3) )/\mathscr{O}_{X}\\
S^{-1} &= \shom(E^2, E^1) \oplus \shom (E^3, E^2)\\
S^{-2} &= \shom(E^3, E^1)
\end{align*}
and the Higgs field $\eta$ is given by $\theta^\vee \otimes 1 + 1 \otimes \theta$, where $\theta^\vee$ is the Higgs field on the Higgs bundle corresponding to $\bv^\vee$ on $X$. Therefore, the Higgs bundle $(T_\star, \omega)$ on $C^\circ$ has $5$ graded pieces as well because we can check the number of graded pieces of $(T_\star, \omega)$ by pulling back along the map $f: f^{-1}(C^\circ) \to C^\circ$ and checking the number of pieces in the Hodge filtration of the corresponding variation of Hodge structure on $f^{-1}(C^\circ)$ (which is simply the pullback of $\End^0(\bv)$ from $X$ to $f^{-1}(C^\circ)$).

Since the irreducible local system $\End^{0}(\bv)$ on $C^\circ$ corresponds to a Higgs bundle $(T_\star, \omega)$ with five graded pieces and every summand of $R^1 f_* \bc$ corresponds to a Higgs bundle with at most two graded pieces, $\End^{0}(\bv)$ is not a summand of $R^1 f_* \bc$ as desired.
\end{proof}
\end{proposition}

We have now established that $H^1(\pi_1^{\orb}(C_f), \ad_{\pgl_3}(\rho)) \cong H^1(X, \End^{0}(\bv))$. This allows us to conclude \autoref{thm:quasi-priviliged implies rigid}.

\begin{proof}[Proof of \autoref{thm:quasi-priviliged implies rigid}]
Throughout, we write $M_{\betti}(C_f^{\orb}, \pgl_3)$ to mean the Betti moduli space of projective representations $\pi_1^{\orb}(C_f) \to \pgl_3(\bc)$. Note that a \emph{good} representation $\varphi:\pi_1^{\orb}(C_f) \to \pgl_3(\bc)$ (in the sense of \cite[top of page 2]{Porti}) is a smooth point of $M_{\betti}(C_f^{\orb}, \pgl_3)$ with tangent space isomorphic to $H^1(\pi_1^{\orb}(C_f), \ad_{\pgl_3}(\varphi))$ (see \cite[Theorem 1.3]{Porti}). The representation $\rho: \pi_1^{\orb}(C_f) \to \pgl_3(\bc)$ realizing the projective factorization of $\bv$ is a good representation due to the image of $\rho$ being Zariski-dense in $\pgl_3(\bc)$ and the center of $\pgl_3(\bc)$ being trivial.

Now, we note that since $\pi_1(X) \to \pi_1^{\orb}(C_f)$ is a surjective group homomorphism, the induced map between the Betti moduli spaces of $\pgl_3$-local systems $M_{\betti}(C_f^{\orb}, \pgl_3) \to M_{\betti}(X, \pgl_3)$ is a closed embedding. We claim that this tells us $\bv$ is rigid. 

If $H^1(X, \End^{0}(\bv))$ is zero, then $\bv$ is in fact cohomologically rigid. So, we assume that $H^1(X, \End^{0}(\bv))$ is nonzero. This implies that $H^1( \pi_1^{\orb}(C_f),\ad_{\pgl_3} ( \rho))$ is nonzero. Smoothness of $M_{\betti}(C_f^{\orb}, \pgl_3)$ at $\rho$ allows us to pick some connected open set $U \subseteq M_{\betti}(C_f^{\orb}, \pgl_3)$ containing $\rho$ whose points are smooth.

The image of this open set under the closed embedding $M_{\betti}(C_f^{\orb}, \pgl_3) \to M_{\betti}(X, \pgl_3)$ contains the projectivization $\bp(\bv)$ of $\bv$ by construction. Let $V$ be this image. Then, we have the equalities
\begin{align*}
\dim V = \dim U = \dim H^1(\pi_1^{\orb}(C_f),\ad_{\pgl_3} ( \rho))
\end{align*}
because the map on Betti moduli spaces is a closed embedding. Furthermore, $\bp(\bv)$ is a smooth point of $M_{\betti}(X, \pgl_3)$. To see this, let $Z$ be the irreducible component of $M_{\betti}(X, \pgl_3)$ containing $V$ (and therefore containing $\bp(\bv)$). Then, 
\begin{align*}
 \dim U = \dim V \leq \dim Z \leq \dim H^1(X, \End^0(\bv)) = \dim H^1(\pi_1^{\orb}(C_f),\ad_{\pgl_3} ( \rho)) = \dim U.
\end{align*}
In other words, the dimension of $Z$ is equal to the dimension of the Zariski tangent space at $\bp(\bv)$. Therefore there is an open neighborhood of smooth points of $M_{\betti}(X, \pgl_3)$ containing $\bp(\bv)$, and all of the points in this open neighborhood have their tangent space of the same dimension as that of $\bp(\bv)$.

Finally, note that $\bp(\bv)$ admits a lift to the $\sl_3$-local system $\bv$. For any smooth family $\bp(\bv_t)$ of $\pgl_3$-local systems with $\bp(\bv_0) = \bp(\bv)$, the obstruction class to lifting the family $o(\bp(\bv_t)) \in H^2(X,\mu_3)$ is locally constant. Hence we can (locally) lift the family to a family of $\sl_3$-local systems on $X$ with the same tangent space dimension. This contradicts the assumption that $\bv$ is quasi-privileged, so in fact $\bv$ must be cohomologically rigid.
\end{proof}

We now complete the proof of \autoref{thm:quasi-privileged implies motivic}.

\begin{proof}[Proof of \autoref{thm:quasi-privileged implies motivic}]
Let $X$ be a smooth projective variety over the complex numbers, and let $\bv \in M_{\betti}(X, 3, 1)$ be a rank $3$ quasi-privileged local system with trivial determinant on $X$. Then by \autoref{prop:privileged galois}, $\bv$ and all of its Galois conjugates underlie pure polarizable complex variations of Hodge structure and are defined over rings of integers of number fields. Now, we break the proof up into cases.

Case $1$: If $\bv$ and all of its Galois conjugates underlie pure polarizable complex variations of Hodge structure of weight one, then \autoref{prop: privileged (1,1,1)} allows us to conclude that $\bv$ is of geometric origin.

Case $2$: If $\bv$ (or one of its Galois conjugates) underlies a pure polarizable complex variation of Hodge structure whose Hodge numbers are $(1,1,1)$ and the monodromy is \emph{not} Zariski-dense in $\sl_3(\bc)$, then by \autoref{prop:sym is motivic} we have that $\bv$ is of geometric origin.

Case $3$: If $\bv$ (or one of its Galois conjugates) underlies a pure polarizable complex variation of Hodge structure whose Hodge numbers are $(1,1,1)$ and the monodromy is Zariski-dense in $\sl_3(\bc)$, then $\bv$ (or its Galois conjugate) projectively factors through an orbicurve. Then since $\bv$ is quasi-privileged, \autoref{thm:quasi-priviliged implies rigid} implies that $\bv$ is actually cohomologically rigid. But then, applying the work of Langer-Simpson (\autoref{thm:rank 3 rigid}) tells us that $\bv$ is of geometric origin because $\bv$ is integral.
\end{proof}

\section{Examples}\label{sec:examples}

Privileged local systems need not be rigid, thus a fortiori quasi-privileged local systems need not be rigid. We exhibit an example of a rank 3 privileged  local system on a smooth projective variety that is non-rigid. This example is heavily inspired by \cite{dejong2022rigidnoncohomologicallyrigidlocal}, and is the example mentioned in \autoref{thm:privileged non-rigid}.

\subsection{Notation}
Following notation in \cite[Section 3]{dejong2022rigidnoncohomologicallyrigidlocal}, for $r \geq 2$ we define the standard representation of the symmetric group $S_{r+1}$ \begin{gather*}
    \Lambda_r = \ker(\bz^{r+1} \xrightarrow{\Sigma} \bz), \; \underline{a} = (a_1,\dots,a_{r+1}) \longmapsto \Sigma(\underline{a}) = \sum a_i \\
    \rho_r: S_{r+1} \rightarrow \gl(\Lambda_r).
\end{gather*}

Let $V_4 = \{e,(12)(34),(13)(24),(14)(23)\} \lhd S_4$ be the Klein four group. Then $V_4$ is the unique normal subgroup such that the quotient is isomorphic to $S_3$. We fix the following explicit surjection, given on a set of generators of $S_4$ \begin{align*}
    q: S_4 &\twoheadrightarrow S_3 \\
    (12) &\mapsto (23) \\
    (23) &\mapsto (12) \\
    (34) &\mapsto (23)
\end{align*}
This is well-defined, surjective since $(12),(23)$ are generators, and $V_4$ is its kernel. From now on, when we write $S_4 \twoheadrightarrow S_3$ we will always mean this map. 

\begin{remark}
    The choice of surjection can be made more natural by considering the permutation action of $S_4$ on the set of sets \[
    A_1=\{\{1,2\},\{3,4\}\},A_2=\{\{1,3\},\{2,4\}\},A_3=\{\{1,4\},\{2,3\}\}.
    \]
    Indeed, $\sigma \in S_4$ acts via permutations of $\{A_1,A_2,A_3\}$ and this induces the map above.
\end{remark}

We will now always consider $\Lambda_2$ as an $S_4$-representation, via $S_4 \twoheadrightarrow S_3 \xrightarrow{\rho_2} \gl(\Lambda_2)$. More precisely, the action of the generators $(12),(23),(34) \in S_4$ on $\Lambda_2 = \{(a_1,a_2,a_3) \in \bz^3 \mid \sum_i a_i = 0 \}$, is given by
\begin{align}\label{formula: S_4 to S_3}
    (12)\cdot (a_1,a_2,a_3)  &=(a_1,a_3,a_2) \nonumber \\ 
    (23)\cdot (a_1,a_2,a_3)  &=(a_2,a_1,a_3) \\
    (34)\cdot (a_1,a_2,a_3)  &=(a_1,a_3,a_2) \nonumber
\end{align}

\begin{remark}
    Note that since $S_3$ has no outer automorphisms, for $\alpha \in \aut(S_3)$, and $\rho$ a $S_3$-representation, we have \[\rho \circ \alpha \cong \rho.\]
    Hence the isomorphism class of any representation of $S_4$ via  $S_4 \twoheadrightarrow S_3$ does not depend on the surjection we choose.
\end{remark}

\subsection{General idea}\label{subsec: general idea}
Inspired by the ideas of \cite{dejong2022rigidnoncohomologicallyrigidlocal} we will look at smooth projective varieties $X$ with fundamental groups of the form \[\pi_1(X,x) = B \rtimes S_4,\] where $B \cong \bz^n$ with an $S_4$-action, such that $H^1(B,\bc^*)^{S_4} = 0$. We want to study the local system given by \begin{align*}
    \bv : \pi_1(X,x) = B \rtimes S_4 \twoheadrightarrow S_4 \xrightarrow{\rho_3} \gl(\Lambda_3 \otimes_\bz \bc).
\end{align*}
Our goal is to choose $B$ such that \begin{enumerate}
    \item There exists a smooth non-constant family through $\bv$.
    \item The tangent space dimension differs from all other local systems.
\end{enumerate}
We will construct $B_1$ satisfying (1), $B_2$ satisfying (2), set $B = B_1 \oplus B_2$, and show the sum satisfies both (1) and (2).

\begin{remark}
    More precisely, in the construction below, we will have $B_1 = \Lambda_2 \otimes_\bz H_1(E,\bz)$, $B_2= \Lambda_3 \otimes_\bz H_1(E,\bz)$ where $E$ is an elliptic curve. Here $S_4$ acts trivially on $H_1(E,\bz)$.
\end{remark}

We establish a few general facts first.
Let $\bw$ be an irreducible rank 3 local system on $X$ with determinant $\det(\bv)$. By Clifford theory, the restriction to $B$ decomposes into characters, i.e. \[\bw|_{B} = \chi_1 \oplus \chi_2 \oplus \chi_3.\] By irreducibility of $\bw$, the $S_4$-action induces a transitive action on the set of (isomorphism classes of) characters. We deduce that either all $\chi_i$ are isomorphic or pairwise non-isomorphic. Further, we use the exact sequence \begin{align*}
    0 &\rightarrow H^1(S_4, H^0(B, \End^{0}(\bw)|_B)) \rightarrow H^1(\pi_1(X,x), \End^{0}(\bw)) \rightarrow H^1(B, \End^{0}(\bw)|_B)^{S_4} \\ &\rightarrow H^2(S_4, H^0(B, \End^{0}(\bw)|_B)) \rightarrow \dots.
\end{align*}
The first and last terms are zero since $S_4$ is finite and $H^0(B, \End^{0}(\bw)|_B)$ is some complex $S_4$-representation. We deduce $\dim H^1(\pi_1(X,x), \End^{0}(\bw)) = \dim H^1(B, \End^{0}(\bw)|_B)^{S_4}$. 

\textbf{Case $1$}: Assume all $\chi_i$ are isomorphic, hence $\chi = \chi_i$ is fixed by $S_4$. Then $H^1(B,\bc^*)^{S_4} = 0$ forces the character to be trivial, thus the representation factors through $S_4$. Note that for fixed determinant there is only one isomorphism class of irreducible local systems factoring through $S_4$ of rank 3, namely $\bv$. The tangent space dimension will be \begin{align*}\dim H^1(X, \End^{0}(\bv)) &= \dim \hom_{S_4}(B \otimes_\bz \bc, \End^{0}(\bv)).\end{align*} As seen in \cite[Example 1]{article} using a character table, we have that as an $S_4$-representation, \[\End(\bv) \cong \End(\rho_3) \cong (\mathbf{1} \oplus \Lambda_3 \oplus \Lambda_2 \oplus (\Lambda_3 \otimes_\bz \sgn))\otimes_\bz \bc\] hence
    \[\End^{0}(\bv) \cong \End^{0}(\rho_3) \cong (\Lambda_3 \oplus \Lambda_2 \oplus (\Lambda_3 \otimes_\bz \sgn))\otimes_\bz \bc.\]

\textbf{Case $2$}: Assume the $\chi_i$ are pairwise non-isomorphic. Since the permutation action of $S_4$ is transitive its action either factors through $S_3$ or $C_3$. But $S_4$ has no quotient isomorphic to $C_3$, hence we deduce that $S_4$ acts on $\{\chi_1,\chi_2,\chi_3\}$ via $S_4 \twoheadrightarrow S_3$ after possibly relabeling the characters. After restricting to $B$, we have 
$\End^{0}(\bw)|_B \cong D \oplus \bigoplus_{i \neq j} \chi_i \chi_j^{-1}$ where $D = \{(a_1,a_2,a_3) \in \bc^3 \mid \sum_i a_i = 0\}$ with trivial $B$ action. Since $B \cong \bz^n$ we have that $H^1(B,\chi_i \chi_j^{-1}) = 0$ for $i\neq j$. Thus we get \begin{align*}H^1(B, \End^{0}(\bw)|_B) = H^1(B, D \oplus \bigoplus_{i \neq j} \chi_i \chi_j^{-1}) = H^1(B, D). \end{align*} Finally, the $B$-action on $D$ is trivial, and the $S_4$-action on the diagonal is induced by the permutation action on $\{\chi_1,\chi_2,\chi_3\}$ via $S_4 \twoheadrightarrow S_3$, hence $D$ is precisely isomorphic to $\Lambda_2\otimes_\bz \bc$. We deduce
\begin{align*}
    \dim H^1(X, \End^{0}(\bw)) &= \dim \hom_{S_4}(B \otimes_\bz \bc, \Lambda_2\otimes_\bz \bc)
\end{align*} 

If we want to distinguish the tangent space dimension from Case $1$ to Case $2$, we can pick \[
B_2 = \Lambda_3 \qquad \text{or} \qquad B_2 = \Lambda_3 \otimes \sgn.
\]
If we want our local system to be in a non-constant family, then the tangent space dimension of all local systems in that family must be non-zero. For this to be possible, we need to ensure the tangent space dimensions in Case 2 are non-zero. For this we pick
\[
B_1 = \Lambda_2.
\] 
In the following two sections, we show that these choices indeed give us a privileged but not rigid local system.

\subsection{Realization as a fundamental group}

We first show that the semidirect product above can be realized as a fundamental group of a smooth projective variety. This construction and notation follow \cite[Section 3]{dejong2022rigidnoncohomologicallyrigidlocal}. Let $E$ be an elliptic curve. We define the abelian varieties \begin{align*}
    A_r = E \otimes_\bz \bz^r = E \times \dots \times E
\end{align*}
of dimension $r$. We choose a basis $(b_i = (e_i - e_{r+1}))_i$ of $\Lambda_r$, then we view $\mathrm{Aut}(E) \otimes_\bz  \mathrm{Aut}(\Lambda_r) \subset \mathrm{Aut}(A_r)$, and $S_{r+1}$ acts on $A_r$ via \begin{align*}
    S_{r+1} \xrightarrow{Id \otimes \rho_r} \mathrm{Aut}(E) \otimes_\bz  \gl(\Lambda_r) \subset \mathrm{Aut}(A_r).
\end{align*}

Let $P_r$ be a smooth projective simply connected variety over $\bc$ on which $S_{r+1}$ acts freely, see \cite[Proposition 15]{Serre2003SureLT}. Define \begin{align*}
    Y = A_3 \times_\bc A_2 \times_\bc P_3 \qquad X = Y/S_4,
\end{align*}
where $S_4$ acts diagonally, via $S_{4} \xrightarrow{Id \otimes \rho_3} \mathrm{Aut}(A_3)$ on $A_3$ and via $S_4 \twoheadrightarrow S_3 \xrightarrow{Id \otimes \rho_2} \mathrm{Aut}(A_2)$ on $A_2$. We then have the Galois cover $f: Y \rightarrow X$. We fix $\bc$-points $y = (0,0,p)$ of $Y$ and $x = f(y)$. We then obtain \begin{align*}
    1\rightarrow \pi_1(Y,y) \rightarrow \pi_1(X,x) \rightarrow S_4\rightarrow 1
\end{align*}
This is the same as the homotopy exact sequence of the smooth projective morphism
$f : X \rightarrow P_3 /S_4$ with fibers isomorphic to $A_3 \times_\bc A_2$ and with section $P_3 /S_4 \xrightarrow{\sim} (0_{A_3} \times_\bc 0_{A_2} \times_\bc P_3)/ S_4 \subset X$. We deduce
 \begin{align*}
    \pi_1(X,x) = \pi_1(Y,y) \rtimes S_4 = (H_1(E,\bz) \otimes_\bz (\Lambda_3 \oplus \Lambda_2) )\rtimes S_4
\end{align*}
where $S_4$ acts via $Id\otimes (\rho_3 \oplus \rho_2)$. We define the local system $\bv$ on $X$ given by \begin{align*}
    \pi_1(X,x) \twoheadrightarrow S_4 \xrightarrow{\rho_3} \gl(\Lambda_3 \otimes_\bz \bc).
\end{align*}
The determinant of $\bv$ is given by \begin{align*}
    \det(\bv): \pi_1(X,x) \twoheadrightarrow S_4 \xrightarrow{\sgn} \{\pm 1\}.
\end{align*}

\noindent
Note that in the notation of \autoref{subsec: general idea}, we have \[ \pi_1(X,x) = B \rtimes S_4,\]
where $B = \pi_1(Y,y) = B_1 \oplus B_2$ for $B_1 = H_1(E,\bz) \otimes_\bz \Lambda_2$ and $B_2 = H_1(E,\bz) \otimes_\bz \Lambda_3$.

\subsection{Proof of \autoref{thm:privileged non-rigid}}

The goal of this subsection is to show the following proposition.
\begin{proposition}\label{privileged non rigid}
    The local system $\bv$ is privileged but not rigid in  $M_{\betti}(X,3, \det(\bv))$.
\end{proposition}

We obtain \autoref{thm:privileged non-rigid} by twisting with a character to get a determinant one local system.
\begin{proof}[Proof of \autoref{thm:privileged non-rigid}]
    The local system $\bv$ (in \autoref{privileged non rigid}) has determinant \begin{align*}
    \det(\bv): \pi_1(X,x) \twoheadrightarrow S_4 \xrightarrow{\sgn} \{\pm 1\}.
\end{align*} We can twist the representation by $\det(\bv)$, then $\det(\bv \otimes \det(\bv)) = \det(\bv)^4 = 1.$ Twisting by $\det(\bv)$ gives an isomorphism from $M_{\betti}(X,3, \det(\bv))$ to $M_{\betti}(X,3, 1)$, in particular, it preserves irreducibility, tangent space dimension, and non-rigidity. Hence $\bv \otimes \det(\bv)$ is a rank 3 privileged non-rigid determinant 1 local system.
\end{proof}

\begin{lemma}\label{lem: no characters}
    We have $H^1(\pi_1(Y,y),\bc^*)^{S_4} = 0$. 
\end{lemma}
\begin{proof}
    Since $H^1(\pi_1(Y,y),\bc^*)^{S_4}= H^1(H_1(E,\bz) \otimes_\bz \Lambda_3,\bc^*)^{S_4} \oplus H^1(H_1(E,\bz) \otimes_\bz \Lambda_2,\bc^*)^{S_4}$ it suffices to show vanishing of both summands.
    For $r \in \{2,3\}$, we have \begin{align*}
        H^1(H_1(E,\bz) \otimes_\bz \Lambda_r,\bc^*)^{S_4} = \hom_{S_4}(H_1(E,\bz) \otimes_\bz \Lambda_r, \bc^*) \cong \hom_{S_4}(\Lambda_r^{\oplus 2}, \bc^*).
    \end{align*} 
    Write $(\Lambda_r)_{S_4} = \Lambda_r/\langle \sigma n - n\mid \sigma \in S_4, n \in \Lambda_r \rangle$ for the coinvariants. Then giving an $S_4$-equivariant character of $\Lambda_r$ is the same as giving a character of $(\Lambda_r)_{S_4}$, i.e. \[
    \hom_{S_4}(\Lambda_r, \bc^*) = \hom((\Lambda_r)_{S_4}, \bc^*)
    \]  

    For $r \in \{2,3\}$, we claim $(\Lambda_r)_{S_4} = 0$. To see this, note that $\Lambda_r = \{(a_1,\dots,a_{r+1}) \mid \sum_i a_i = 0\}$ is generated by $e_i - e_j$, $i\neq j \in \{1,\dots,r+1\}$, where $e_i \in \bz^{r +1}$ is $(0,\dots,1,\dots,0)$ with $1$ at the $i$-th position. Choose $k\neq i,j$, then for $\sigma = (i\;j) \in S_{r+1}$ we have $\sigma(e_k - e_i) - (e_k - e_i)= e_i - e_j$. Hence \begin{align*}
        (\Lambda_3)_{S_4} &= \Lambda_3/\langle \sigma n - n\mid \sigma \in S_4, n \in \Lambda_3 \rangle = 0 \\
        (\Lambda_2)_{S_4} &= (\Lambda_2)_{S_3} = \Lambda_2/\langle \sigma n - n\mid \sigma \in S_3, n \in \Lambda_2 \rangle =0
    \end{align*}
    since $S_4$ acts on $\Lambda_2$ via $S_3$.
\end{proof}

\begin{lemma}\label{lem: smooth family}
    There exists a smooth non-constant family \begin{align*}
        \Psi_{s,t}: (H_1(E,\bz) \otimes_\bz \Lambda_2) \rtimes S_4 \rightarrow \gl_3(\bc)
        \end{align*}
        where $(s,t) \in \bg_m \times \bg_m$, of irreducible representations, such that $\Psi_{1,1}$ is isomorphic to the representation defined by \[
        (H_1(E,\bz) \otimes_\bz \Lambda_2) \rtimes S_4 \twoheadrightarrow S_4 \xrightarrow{\rho_3} \gl_3(\Lambda_3 \otimes_\bz \bc)
        \]
        and $\det{\Psi_{s,t}} = \det{\Psi_{1,1}}$.
\end{lemma}
\begin{proof}
    First, we choose a basis of $\Lambda_3 \otimes_\bz \bc = \{(a_1,a_2,a_3,a_4) \in \bc^4 \mid \sum_i a_i = 0 \}$ so we can identify $\gl_3(\Lambda_3 \otimes_\bz \bc ) \cong \gl_3(\bc)$. Since $V_4$ is abelian, we can choose a basis such that $\rho_3|_{V_4}$ is diagonal. We choose the basis \begin{align*}
    v_1&=e_1+e_2-e_3-e_4,\\
    v_2&=e_1-e_2+e_3-e_4,\\
    v_3&=e_1-e_2-e_3+e_4
    \end{align*}
    where $e_i$ are the standard basis vectors in $\bc^4$. 
    Then $S_4$ permutes the one-dimensional subspaces spanned by the basis vectors via $S_4 \twoheadrightarrow S_3$, more precisely, \begin{align}\label{eq: permutation basis}
        \sigma \cdot v_i = \lambda_{\sigma,i} \cdot v_{q(\sigma)(i)}
    \end{align}
    where $q(\sigma)$ is its image under $q: S_4 \twoheadrightarrow S_3$ and $\lambda_{\sigma,i} = \pm 1$. We then define \begin{align}\label{map: basis}
         \Lambda_3 \otimes_\bz \bc 
        &\xrightarrow{\sim} \bc^3 \\
        v_i &\mapsto e_i \nonumber
    \end{align}
    where the $e_i$ on the right-hand side are the standard basis vectors of $\bc^3$. This induces  $\gl_3(\Lambda_3 \otimes_\bz \bc ) \cong \gl_3(\bc)$ and we view \begin{align}\label{map: basis2}\rho_3: S_4 \rightarrow \gl_3(\bc)\end{align} via this identification.

    Next, for $t \in \bg_m$ we define a family of representations \begin{align*}\psi_t: \Lambda_2= \{(a_1,a_2,a_3) \in \bz^3 \mid \sum_i a_i = 0 \} &\rightarrow \gl_3(\bc) \\ (a_1,a_2,a_3) &\mapsto \diag(t^{a_1},t^{a_2},t^{a_3}).\end{align*} After choosing a basis of $H_1(E,\bz)$ we view $H_1(E,\bz) \otimes_\bz \Lambda_2 \cong \Lambda_2 \oplus \Lambda_2$ as an $S_4$-representation.
    We claim that for $(s,t) \in \bg_m\times\bg_m$,  \begin{align*}
        \Psi_{s,t}: (H_1(E,\bz) \otimes_\bz \Lambda_2) \rtimes S_4 &\rightarrow \gl_3(\bc) \\
        ((\underline{a},\underline{b}), \sigma) &\mapsto \psi_s(\underline{a})\psi_t(\underline{b})\rho_3(\sigma)
    \end{align*}
    is a well-defined representation. To prove that, we need to show the identity  \begin{align}\label{eq: matrices}
        \rho_3(\sigma)\psi_s(\underline{a})\rho_3(\sigma)^{-1} = \psi_s(\sigma(\underline{a}))
    \end{align}
    for $\sigma \in S_4$, $\underline{a} \in 
    \Lambda_2$. \autoref{eq: matrices} will roughly follow since the permutation action of $S_4$ on the one-dimensional subspaces spanned by $v_1,v_2,v_3$ is the same as on $(a_1,a_2,a_3) \in \Lambda_2$. More precisely, for $\sigma \in S_4$, $\underline{a} = (a_1,a_2,a_3) \in \Lambda_2$, $j = q(\sigma)^{-1}(i)$ we have 
    \begin{align*}
    \rho_3(\sigma)\psi_s(\underline{a})\rho_3(\sigma)^{-1}(v_i) &=
    \rho_3(\sigma)\psi_s(\underline{a})\lambda_{\sigma,i}^{-1} v_{j} \\ &=\rho_3(\sigma)s^{a_j}\lambda_{\sigma,i}^{-1} v_{j}
    \\ &=s^{a_j}v_i
    \\ &=\psi_s(\sigma(\underline{a}))v_i
    \end{align*}

    \noindent
    Then \begin{align*}
        \Psi_{s,t}(((\underline{a},\underline{b}), \sigma))\Psi_{s,t}(( (\underline{a'},\underline{b'}), \sigma')) &= \psi_s(\underline{a})\psi_t(\underline{b})\rho_3(\sigma)\psi_s(\underline{a'})\psi_t(\underline{b'})\rho_3(\sigma') \\ 
        &= \psi_s(\underline{a})\psi_t(\underline{b})\psi_s(\sigma(\underline{a'}))\psi_t(\sigma(\underline{b'}))\rho_3(\sigma)\rho_3(\sigma') \\ 
    & = \psi_s(\underline{a} + \sigma(\underline{a'}))\psi_t(\underline{b} + \sigma(\underline{b'}))\rho_3(\sigma)\rho_3(\sigma') \\
    &=\Psi_{s,t}(((\underline{a},\underline{b}), \sigma)\cdot( (\underline{a'},\underline{b'}), \sigma')),
    \end{align*}
    so the $\Psi_{s,t}$ are well-defined representations.

To show irreducibility, it suffices to show irreducibility on the subgroup $S_4 < (H_1(E,\bz) \otimes_\bz \Lambda_2) \rtimes S_4$, where we identify $S_4 \cong \{(0,\sigma) \in (H_1(E,\bz) \otimes_\bz \Lambda_2) \rtimes S_4\}$. By definition of $\Psi_{s,t}$, restricted to $S_4 < (H_1(E,\bz) \otimes_\bz \Lambda_2) \rtimes S_4$, it is isomorphic to $\rho_3$, hence irreducible.

To show that they are non-isomorphic representations, we compute their traces \begin{align*}
        \tr \Psi_{s,t}(((1,-1,0),(0,0,0)), (12)) &= \tr (\psi_s(1,-1,0)\psi_t(0,0,0)\rho_3((12))) \\ &=\tr 
        (\begin{pmatrix}
        s&0&0\\
        0&s^{-1}&0\\
        0&0&1
        \end{pmatrix}
        \begin{pmatrix}
        1&0&0\\
        0&1&0\\
        0&0&1
        \end{pmatrix}
        \begin{pmatrix}
        1&0&0\\
        0&0&-1\\
        0&-1&0
        \end{pmatrix})
        \\ &=\tr 
        \begin{pmatrix}
        s&0&0\\
        0&0&-s^{-1}\\
        0&-1&0
        \end{pmatrix}
         \\ &= s
        \\ \tr \Psi_{s,t}(((0,0,0),(1,-1,0)), (12)) &= \tr (\psi_s(0,0,0)\psi_t(1,-1,0)\rho_3((12))) \\ &=\tr 
        (\begin{pmatrix}
        1&0&0\\
        0&1&0\\
        0&0&1
        \end{pmatrix}
        \begin{pmatrix}
        t&0&0\\
        0&t^{-1}&0\\
        0&0&1
        \end{pmatrix}
        \begin{pmatrix}
        1&0&0\\
        0&0&-1\\
        0&-1&0
        \end{pmatrix})
        \\ &=\tr 
        \begin{pmatrix}
        t&0&0\\
        0&0&-t^{-1}\\
        0&-1&0
        \end{pmatrix}
         \\ &= t.
    \end{align*} Hence for $(s,t) \neq (s',t')$ we have $\Psi_{s,t}$ is not isomorphic to $\Psi_{s',t'}$. 
    Finally, $\det \Psi_{s,t}(((\underline{a},\underline{b}), \sigma)) = s^{\sum a_i} t^{\sum b_i} \sgn(\sigma) = \det \Psi_{1,1}(((\underline{a},\underline{b}), \sigma))$.
\end{proof}

\begin{proof}[Proof of \autoref{privileged non rigid}]
    We first show that there is a smooth family $\bv_{s,t}$ in $M_{\betti}(X,3, \det(\bv))$ such that $\bv_{1,1} = \bv$. Using the family $\Psi_{s,t}$ from \autoref{lem: smooth family}, we get a smooth non-constant family of irreducible representations \begin{align*}\bv_{s,t}: \pi_1(X,x)\twoheadrightarrow (H_1(E,\bz) \otimes_\bz \Lambda_2) \rtimes S_4 \xrightarrow{\Psi_{s,t}} \gl_3(\bc),\end{align*} with $\bv_{1,1} \cong \bv$. 

    In order to show $\bv$ is privileged, we need to compute the Zariski tangent space of every local system in $M_{\betti}(X,3, \det(\bv))$. Let $\bw \in M_{\betti}(X,3, \det(\bv))$. By considerations in \autoref{subsec: general idea}, after restricting to $\pi_1(Y,y)$, we have $\bw|_{\pi_1(Y,y)} = \chi_1 \oplus \chi_2 \oplus \chi_3$.

    In the case where the $\chi_i$ are isomorphic, $S_4$ fixes the character. By \autoref{lem: no characters} there are no non-trivial characters of $\pi_1(Y,y)$ fixed by $S_4$, hence $\chi$ is the trivial character. Then $\bw$ factors through $S_4$, and there is only one irreducible rank 3 representation of determinant $\det(\bv)$, i.e. \[\bw \cong \bv.\] By Case $1$ in \autoref{subsec: general idea} we then get
    \begin{align*}
        \dim H^1(X, \End^{0}(\bv)) &= \dim \hom_{S_4}(\pi_1(Y,y) \otimes_\bz \bc, \End^{0}(\bv)) \\ &= \dim \hom_{S_4}((\Lambda_3 \oplus \Lambda_2)^{\oplus 2}\otimes_\bz \bc, (\Lambda_3 \oplus \Lambda_2 \oplus (\Lambda_3 \otimes_\bz \sgn))\otimes_\bz \bc) \\ &= 4.
    \end{align*}
    
    \noindent
    In the case where the $\chi_i$ are pairwise non-isomorphic, we have by Case $2$ in \autoref{subsec: general idea} \begin{align*}
        \dim H^1(X, \End^{0}(\bw)) &= \dim \hom_{S_4}(\pi_1(Y,y) \otimes_\bz \bc, \Lambda_2 \otimes_\bz \bc)  \\ &= \dim \hom_{S_4}((\Lambda_3 \oplus \Lambda_2)^{\oplus 2}\otimes_\bz \bc, \Lambda_2 \otimes_\bz \bc) \\ &= 2.
    \end{align*} 
    Note that $\hom_{S_4}(\Lambda_3 \otimes_\bz \bc, \Lambda_2 \otimes_\bz \bc) = 0$ by Schur's Lemma since they are both irreducible complex non-isomorphic representations of $S_4$. 

    Thus $\bv$ is the unique point with tangent
    dimension 4, and every other point has tangent dimension 2. Hence $\bv$ is privileged. The non-constant family $\bv_{s,t}$ proves
    that it is not rigid.
\end{proof}

 \begin{remark}
    As remarked in \cite[Introduction]{dejong2022rigidnoncohomologicallyrigidlocal},
    using the external tensor product of local systems we can attempt to create privileged but not rigid local systems of higher rank or of infinite monodromy. For instance, for an example of higher rank we could look at $\bv \boxtimes \bv$ on $X \times_\bc X$. However, the precise construction in this paper cannot be repeated with the example in \cite{dejong2022rigidnoncohomologicallyrigidlocal} for rank $r \geq 4$. This is because $S_{r+1}$ for $r \geq 4$ has no nontrivial proper quotients other than $C_2$, and we crucially use the fact that $S_3$ is a quotient of $S_4$.
 \end{remark}

\bibliographystyle{alpha}
\bibliography{bibliography-privileged.bib}

\end{document}